\documentclass[a4paper,12pt]{article}

\usepackage[left=2cm,right=2cm, top=2cm,bottom=2cm,bindingoffset=0cm]{geometry}

\usepackage{verbatim}
\usepackage{amsmath}
\usepackage{amsthm}
\usepackage{amssymb}
\usepackage{delarray}
\usepackage{cite}
\usepackage{mathrsfs}
\usepackage{graphicx}

\newcommand{\ga}{\gamma}
\newcommand{\de}{\delta}
\newcommand{\la}{\lambda}
\newcommand{\om}{\omega}

\newcommand{\eps}{\varepsilon}

\newcommand{\iy}{\infty}

\theoremstyle{plain}

\newtheorem{thm}{Theorem}
\newtheorem{lem}{Lemma}

\theoremstyle{definition}

\theoremstyle{remark}

\newtheorem{remark}{Remark}

 \allowdisplaybreaks

\begin{document}
\begin{center}
{\large\bf A partial inverse problem for the Sturm-Liouville operator on a star-shaped graph
}
\\[0.2cm]
{\bf Natalia P. Bondarenko} \\[0.2cm]
\end{center}

\vspace{0.5cm}

{\bf Abstract.} The Sturm-Liouville operator on a star-shaped graph is considered. We assume that the potential is known 
a priori on all the edges except one, and study the partial inverse problem, which consists in recovering the potential
on the remaining edge from the part of the spectrum. A constructive method is developed for the solution
of this problem, based on the Riesz-basicity of some sequence of vector functions. The local solvability of the inverse problem
and the stability of its solution are proved.

\medskip


\medskip


\vspace{1cm}

{\bf \large 1. Introduction}

\bigskip

Differential operators on geometrical graphs (also called quantum graphs) models different structures in
organic chemistry, mesoscopic physics, nanotechnology, microelectronics, acoustics and other fields of science and
engineering (see \cite{Mont70, Nic85, LLS94, Kuch02, Exner08} and references therein). In the recent years, spectral problems on quantum graphs
attract much attention of mathematicians. The reader can find the results on {\it direct problems} of studying properties of the spectrum and root functions, for example,
in \cite{NS00, Kuch05, PPP04}. {\it Inverse problems} consist in constructing differential operators by their spectral characteristics.
Inverse problems for quantum graphs were studied in \cite{Bel04, Yur05, Yur09, Yur10-1, Yur10-2, BW05, KN05, Piv07}. 
We focus our attention on the so-called coefficient inverse problems, which consist in recovering coefficients of differential equations (i.e. potentials
of Sturm-Liouville equations) on the edges of the graph, while the structure of the graph and the matching conditions in the vertices are known a priori.
Such problems generalize the classical inverse spectral problems on a finite interval (see the monographs \cite{Mar77, Lev84, PT87, FY01}).

In this paper, we consider a {\it partial} inverse problem for the Sturm-Liouville operator on a graph. 
The potential is supposed to be known a priori on some part of the graph, and is being recovered on the remaining part. 
We know only a few results in this direction. V.N.~Pivovarchick \cite{Piv00} considered the star-shaped graph with three edges, 
and assumed that the potential is known on one or two edges. 
He proved that two or one spectra, respectively, uniquely determine the potential on the remaining edges. Later C.-F.~Yang \cite{Yang10}
showed that some fractional part of the spectrum is sufficient to determine the potential on one edge of the star-shaped graph,
if the potential is known on the other edges. The case, when the potential is unknown only on a part of one edge, was considered in the papers 
\cite{Yang10, Yang11, Yur09-2}.
However, the authors of the mentioned papers proved only uniqueness theorems, and presented neither algorithms for solution 
nor sufficient conditions for the solvability of the partial inverse problems.

Another type of partial inverse problems was studied in \cite{BS17}. We considered the Sturm-Liouville operator on the tree, and showed that
if the potential is given on one edge, then we need one spectrum less to recover the potential on the whole graph, comparing with 
the full inverse problem (for example, \cite{Yur05}).

We note that partial inverse problems on graphs are related with the Hochstadt-Lieberman problem \cite{HL78, GS00, Sakh01, MP10} on a finite interval:
the potential is given on the half of the interval, recover the potential on the other half by one spectrum.
However, the methods developed for the Hochstadt-Lieberman problem are not always applicable to partial inverse problems on graphs.

In this paper, we consider the Sturm-Liouville operator on the star-shaped graph, and suppose the potential is known on all the edges except one.
We study the partial inverse problem, formulated by C.-F.~Yang \cite{Yang10}. We develop the constructive algorithm for recovering 
the potential from the part of the spectrum. In addition, the local solvability of the inverse problem and the stability of the solution are proved.
Our method is based on the Riesz-basicity of some system of vector functions. We hope that the ideas of this paper will be useful for 
investigation of more complicated partial inverse problems on graphs, and finally help to find the minimal data, determining the quantum graph.

\bigskip 

{\bf \large 2. Problem statement} 

\bigskip

Consider a star-shaped graph $G$ with edges $e_j$, $j = \overline{1, m}$, of equal length $\pi$. 
For each edge $e_j$, introduce the parameter $x_j \in [0, \pi]$. The value $x_j = 0$ corresponds to the boundary vertex,
associated with $e_j$, and $x_j = \pi$ corresponds to the internal vertex.

We study the boundary value problem $L$ for the system of the Sturm-Liouville equations on the graph $G$:
\begin{equation} \label{eqv}
    -y''_j(x_j) + q_j(x_j) y_j(x_j) = \la y_j(x_j), \quad x_j \in (0, \pi), \quad j = \overline{1, m},
\end{equation}
with the Dirichlet conditions in the boundary vertices
\begin{equation} \label{BC}
    y_j(0) = 0, \quad j = \overline{1, m},
\end{equation}
and the standard matching conditions in the internal vertex
\begin{equation} \label{cont}
  	y_1(\pi) = y_j(\pi), \quad j = \overline{2, m}, 
\end{equation}
\begin{equation} \label{kirch}
    \sum_{j = 1}^m y_j'(\pi) = 0.
\end{equation}
The functions $q_j$ in \eqref{eqv} are real-valued
and belong to $L_2(0, \pi)$. We refer to them as {\it the potentials} on the edges $e_j$.

The eigenvalues of $L$ are described by the following lemma. This result was obtained by V.N.~Pivovarchick \cite{Piv00}
for $m = 3$, and can be generalized for an arbitrary $m$.

\begin{lem}
The boundary value problem $L$ has a countable set of real eigenvalues, which can be numbered 
as $\{ \la_{nk} \}_{n \in \mathbb N, \, k = \overline{1, m} }$ (counting 
with their multiplicities) to satisfy the following asymptotics formulas
\begin{align} \label{asymptrho1}
    \rho_{n1} & =  n - \frac{1}{2} + \frac{\hat \om}{\pi n} + \frac{\varkappa_{n1}}{n}, \\ \label{asymptrho2}
    \rho_{nk} & = n + \frac{z_{k-1}}{\pi n} + \frac{\varkappa_{nk}}{n}, \quad k = \overline{2, m},
\end{align}
where $\rho_{nk} = \sqrt{\la_{nk}}$, $\{ \varkappa_{nk} \}_{n \in \mathbb N} \in l_2$, $k = \overline{1, m}$, 
$\hat \om = \frac{1}{m} \sum\limits_{j = 1}^m \om_j$, $\om_j = \frac{1}{2} \int\limits_0^{\pi} q_j(x) \, dx$,
and $z_k$, $k = \overline{1, m-1}$, are the roots of the characteristic polynomial 
$$
    P(z) = \frac{d}{d z} \prod_{k = 1}^m (z - \om_k),
$$
counting with their multiplicities.
\end{lem}

In this paper, we solve the following partial inverse problem.

\medskip

{\bf IP}. Given the potentials $q_j$, $j = \overline{2, m}$, and the sequence
$\{ \la_{nk}\}_{n \in \mathbb N, \, k = 1, 2}$ of the eigenvalues of $L$. Find the potential $q_1$.

\medskip

In view of the symmetry, one can change the potential $q_1$ to arbitrary $q_j$, $j = \overline{2, m}$, and 
the eigenvalues $\{ \la_{n2} \}_{n \in \mathbb N}$ to a sequence $\{ \la_{nk} \}_{n \in \mathbb N}$ with an arbitrary fixed $k = \overline{3, m}$
in the problem statement.

The uniqueness theorem for the IP was proved in \cite{Yang10}. Note that for $m = 2$ the IP turns into the standard Hochstadt-Lieberman
problem \cite{HL78}.

\bigskip

{\bf \large 3. Solution of the IP} 

\bigskip

Let $S_j(x_j, \la)$ be solutions of equations \eqref{eqv}, satisfying the initial conditions
$S_j(0, \la) = 0$, $S'_j(0, \la) = 1$, $j = \overline{1, m}$. The eigenvalues of $L$ coincide with the zeros
of {\it the characteristic function}, which can be represented in the form
\begin{equation} \label{defDelta}
   \Delta(\la) = S_1'(\pi, \la) \prod_{j = 2}^m S_j(\pi, \la) + S_1(\pi, \la) \left( \sum_{j = 2}^m S_j'(\pi, \la) 
   \prod_{\substack{k = 2 \\ k \ne j}}^m S_k(\pi, \la) \right).
\end{equation}

Impose the following {\bf assumptions}:

\smallskip

(i) all the eigenvalues $\{ \la_{nk} \}_{n \in \mathbb N, \, k = 1, 2}$ are distinct.

(ii) $\la_{nk} > 0$, $n \in \mathbb N$, $k = 1, 2$.

(iii) $S_j(\pi, \la_{nk}) \ne 0$, $j = \overline{1, m}$, $n \in \mathbb N$, $k = 1, 2$.

(iv) $z_1 \ne \om_j$, $j = \overline{1, m}$.

(v) $S_1(\pi, 0) \ne 0$, $S_1'(\pi, 0) \ne 0$.

\smallskip

One can achieve the conditions (ii) and (v) by a shift $q_j \to q_j + C$, $j = \overline{1, m}$.
A root $z_1$ of the characteristic polynomial can be chosen to satisfy (iv), if not all $\om_j$ are equal to each other.

Substituting $\la = \la_{nk}$, $n \in \mathbb N$, $k = 1, 2$, into \eqref{defDelta} and taking the assumption (iii) into account, one
can easily derive the following relation
\begin{equation} \label{relS}
    -\frac{S_1'(\pi, \la_{nk})}{S_1(\pi, \la_{nk})} = \sum_{j = 2}^m \frac{S_j'(\pi, \la_{nk})}{S_j(\pi, \la_{nk})}.
\end{equation}
Note that $M_j(\la) := -\dfrac{S_j'(\pi, \la_{nk})}{S_j(\pi, \la_{nk})}$ is {\it the Weyl function}
for the Sturm-Liouville problem $L_j$ on each fixed edge $e_j$:
\begin{equation} \label{Lj}
  	-y_j''(x_j) + q_j(x_j) y_j(x_j) = \la y_j(x_j), \quad y_j(0) = y_j(\pi) = 0.
\end{equation}
The Weyl functions $M_j(\la)$ are meromorphic, their poles are simple and coincide with the eigenvalues of the boundary value problems $L_j$.
The potential $q_j$ can be uniquely recovered from its Weyl function $M_j(\la)$ by the classical methods (see \cite{Mar77}, \cite{FY01}).

Let the potentials $q_j$, $j = \overline{2, m}$, and the eigenvalues $\{ \la_{nk} \}_{n \in \mathbb N, \, k = 1, 2}$ of the problem $L$ be given.
Rewrite the relation \eqref{relS} in the following form
\begin{equation} \label{relM}
   M_1(\la_{nk}) = - \sum_{j = 2}^m M_j(\la_{nk}) =: g_{nk}, \quad n \in \mathbb N, \quad k = 1, 2.
\end{equation}
The Weyl functions $M_j(\la)$, $j = \overline{2, m}$, can be constructed by the given potentials $q_j$.
Thus, the values $g_{nk} = M_1(\la_{nk})$, $n \in \mathbb N$, $k = 1, 2$, are known, and we have to interpolate 
the meromorphic function $M_1(\la)$ by these values. The subsequences $\{ \la_{n1} \}_{n \in \mathbb N}$
and $\{ \la_{n2} \}_{n \in \mathbb N}$ are asymptotically ``close'' to the zeros and the poles of $M_1(\la)$, respectively.  
However, in view of the assumption (iii), the values $\{ \la_{nk} \}_{n \in \mathbb N, \, k = 1, 2}$ do not coincide 
with the poles of $M_1(\la)$.

The solution $S_1(x, \la)$ can be represented in terms of the transformation operator \cite{Mar77, FY01}:
$$
  	S_1(x, \la) = \frac{\sin \rho \pi}{\rho} + \int_0^{\pi} \mathcal{K}(x, t) \frac{\sin \rho t}{\rho} \, dt, \quad \rho := \sqrt \la.
$$
Using integration by parts and differentiation by $x$, one can easily derive the relations
\begin{align} \label{intS}
   S_1(\pi, \la) = \frac{\sin \rho \pi}{\rho} - \frac{\om \cos \rho \pi}{\rho^2} + \frac{1}{\rho^2} \int_0^{\pi} K(t) \cos \rho t \, dt, \\ \label{intSp}
   S_1'(\pi, \la) = \cos \rho \pi + \frac{\om \sin \rho \pi}{\rho} + \frac{1}{\rho} \int_0^{\pi} N(t) \sin \rho t \, dt,
\end{align}    
where 
$$
\om = \om_1 = \mathcal{K}(\pi, \pi), \quad K(t) = \frac{d}{d t}\mathcal{K}(\pi, t),
\quad N(t) = \frac{\partial}{\partial x} \mathcal{K}(x, t)_{|x = \pi}.
$$
The functions $K(t)$ and $N(t)$ are real and belong to $L_2(0, \pi)$. Note that
$\om_j$, $j = \overline{2, m}$,
can be found by the given $q_j$, $\hat \om$ can be determined from the subsequence $\{ \la_{n1} \}$ (see asymptotic formula \eqref{asymptrho1}),
and $\om_1 = m \hat \om -\sum\limits_{j = 2}^m \om_j$, so the number $\om$ is known.

Substituting \eqref{intS} and \eqref{intSp} into the relation $g_{nk} = -\dfrac{S_1'(\pi, \la_{nk})}{S_1(\pi, \la_{nk})}$, we get
\begin{multline} \label{relNK}
   \frac{1}{\rho_{nk}} \int_0^{\pi} N(t) \sin \rho_{nk} t \, dt + \frac{g_{nk}}{\rho_{nk}^2} \int_0^{\pi} K(t) \cos \rho_{nk} t \, dt \\ =
   - \cos \rho_{nk} \pi - (\om + g_{nk}) \frac{\sin \rho_{nk}}{\rho_{nk}} \pi + \frac{\om g_{nk}}{\rho_{nk}^2} \cos \rho_{nk} \pi.
\end{multline}
Introduce the notation
\begin{equation} \label{defv}
    v_{n-1,1}(t) = \begin{bmatrix} \sin \rho_{n1} t \\ \frac{g_{n1}}{\rho_{n1}} \cos \rho_{n1} t \end{bmatrix}, \quad
    v_{n 2}(t) = \begin{bmatrix} \frac{\rho_{n2}}{g_{n2}} \sin \rho_{n2} t \\ \cos \rho_{n2} t \end{bmatrix}, \quad
    f(t) = \begin{bmatrix} N(t) \\ K(t) \end{bmatrix}, 
\end{equation}
\begin{align} \label{deff1}
    f_{n-1,1} & = - \rho_{n1} \cos \rho_{n1} \pi - (\om + g_{n1}) \sin \rho_{n1} \pi + \frac{\om g_{n1}}{\rho_{n1}} \cos \rho_{n1} \pi, \\ \label{deff2}
    f_{n2}  & = - \frac{\rho_{n2}^2}{g_{n2}} \cos \rho_{n2} \pi - \frac{(\om + g_{n2}) \rho_{n2}}{g_{n2}}\sin \rho_{n2} \pi + \om \cos \rho_{n2} \pi.
\end{align}
For simplicity, we assume here that $g_{n2} \ne 0$. In view of Lemma~\ref{lem:asymptv}, the equality $g_{n2} = 0$ can 
hold only for a finite numbers of values $n \in \mathbb N$,
and this case requires minor modifications.

The vector functions $v_{nk}$ and $f$ belong to the real Hilbert space $\mathcal{H} := L_2(0, \pi) \oplus L_2(0, \pi)$.
The scalar product and the norm in $\mathcal H$ are defined as follows
$$
   (g, h)_{\mathcal H} = \int_0^{\pi} ( g_1(t) h_1(t) + g_2(t) h_2(t)) \, dt, \quad
   \| g \|_{\mathcal H} = \sqrt{\int_0^{\pi} (g_1^2(t) + g_2^2(t)) \, dt}, 
$$
$$
   g = \begin{bmatrix} g_1 \\ g_2 \end{bmatrix}, \quad h = \begin{bmatrix} h_1 \\ h_2 \end{bmatrix}, \quad g, h \in \mathcal H.    
$$
Then \eqref{relNK} can be rewritten in the form
\begin{equation} \label{scal}
    (f, v_{nk})_{\mathcal H} = f_{nk}, 
\end{equation}
where $n \in \mathbb N_0$ ($\mathbb N_0 := \mathbb N \cup \{ 0 \}$), $k = 1$ and $n \in \mathbb N$, $k = 2$.

Note that
$$
   \int_0^{\pi} K(t) \, dt = \int_0^{\pi} \frac{d}{dt}\mathcal{K}(\pi, t) \, dt = \om,
$$
since $\mathcal{K}(\pi, 0) = 0$. Put
\begin{equation} \label{defvf0}
    v_{02} = \begin{bmatrix} 0 \\ 1 \end{bmatrix}, \quad f_{02} = \om.
\end{equation} 
Then the relation \eqref{scal} also holds for $n = 0$, $k = 2$.

Using the relations \eqref{intS}, \eqref{intSp}, \eqref{defv}, \eqref{deff1}, \eqref{deff2} and the asymptotics \eqref{asymptrho1}, \eqref{asymptrho2}, we obtain the
following result.

\begin{lem} \label{lem:asymptv}
The following sequences belong to $l_2$: 
$$
  	\{ g_{n1} - (\hat \om - \om) \}_{n \in \mathbb N}, \quad 
  	\left\{ \tfrac{g_{n2}}{n^2} - (\om - z_1) \right\}_{n \in \mathbb N}, 
  	\quad \{ f_{nk} \}_{n \in \mathbb N_0, \, k = 1, 2}, \quad 
  	\{ \| v_{nk} - v_{nk}^0 \|_{\mathcal H} \}_{n \in \mathbb N_0, \, k = 1, 2}. 
$$
Here and below
$$
    v_{n1}^0 = \begin{bmatrix} \sin (n + 1/2) t \\ 0 \end{bmatrix}, \quad
    v_{n2}^0 = \begin{bmatrix} 0 \\ \cos n t \end{bmatrix}, \quad n \in \mathbb N_0.
$$
\end{lem}

\begin{lem} \label{lem:riesz}
The vector functions $\{ v_{nk} \}_{n \in \mathbb N_0, \, k = 1, 2}$ form a Riesz basis in $\mathcal H$.
\end{lem}

\begin{proof}
By virtue of Lemma~\ref{lem:asymptv}, the system $\{ v_{nk} \}_{n \in \mathbb N_0, \, k = 1, 2}$ is $l_2$-close to the
orthogonal basis $\{ v_{nk}^0 \}_{n \in \mathbb N_0, \, k = 1, 2}$, such that $\| v_{nk}^0 \|_{\mathcal H} = \frac{2}{\pi}$ for $n \in \mathbb N$, $k = 1, 2$.

Let us prove that the system $\{ v_{nk} \}$ is complete in $\mathcal H$. 
Suppose that, on the contrary, there exists an element of $\mathcal H$, orthogonal to all $\{ v_{nk} \}$. Then the relations
$$
    \int_0^{\pi} ( h_1(t) \rho_{nk} \sin \rho_{nk} t + h_2(t) g_{nk} \cos \rho_{nk}t) \, dt = 0, \quad n \in \mathbb N, \, k = 1, 2, \quad
   \int_0^{\pi} h_2(t) \, dt = 0
$$
hold for some functions $h_1, h_2 \in L_2(0, \pi)$. Recall that $g_{nk} = -\dfrac{S_1'(\pi, \la_{nk})}{S_1(\pi, \la_{nk})}$
and $S_1(\pi, \la_{nk}) \ne 0$. Then the function 
$$
   H(\la) := \int_0^{\pi} (h_1(t) S_1(\pi, \la) \rho \sin \rho t - h_2(t) S_1'(\pi, \la) \cos \rho t) \, dt
$$
has zeros at the points $\la = \la_{nk}$, $n \in \mathbb N$, $k = 1, 2$, and $\la = 0$. Clearly, 
the function $H(\la)$ is entire and satisfies the estimate $H(\la) = O(\exp(2 |\mbox{Im}\, \rho| \pi))$.

Denote
$$
    D(\la) := \pi \la \prod_{n = 1}^{\iy} \frac{\la_{n1} - \la}{(n - 1/2)^2} \prod_{s = 1}^{\iy} \frac{\la_{s2} - \la}{s^2}.
$$
The construction of $D(\la)$ resembles representation of characteristic functions of Sturm-Liouville operators in form of infinite products 
(see, for example, \cite[par. 1.1.1]{FY01}).
Note that $D(\la) \sim \rho \cos \rho \pi \sin \rho \pi$ as $|\rho| \to \iy$. Moreover, 
it satisfies the estimate $|D(\la)| \ge C_{\de} |\rho| \exp(2 |\tau | \pi)$ for $\rho$
in $G_{\de} := \{ \rho \colon
|\rho| \ge \de, \: |\rho - \rho_{nk}| \ge \de, \: n \in \mathbb N, \, k = 1, 2 \}$, where $\de > 0$ is a fixed sufficiently small number. 
Consequently, the function
$\dfrac{H(\la)}{D(\la)}$ is entire and $\dfrac{H(\la)}{D(\la)} = O(\rho^{-1})$, $|\rho| \to \iy$ in $G_{\de}$.
By virtue of Liouville's theorem, $H(\la) \equiv 0$. Hence
\begin{equation} \label{Hzero}
   \int_0^{\pi} (h_1(t) S_1(\pi, \la) \rho \sin \rho t - h_2(t) S_1'(\pi, \la) \cos \rho t) \, dt \equiv 0.
\end{equation}

Let $\{ \mu_n \}_{n \in \mathbb N}$ and $\{ \nu_n \}_{n \in \mathbb N_0}$ be the zeros of $S_1(\pi, \la)$ and $S_1'(\pi, \la)$, respectively,
and $\mu_0 = 0$. Substituting $\la = \mu_n$ and $\la = \nu_n$ into \eqref{Hzero} and using the fact, that $S_1'(\pi, \mu_n) \ne 0$
and $S_1(\pi, \nu_n) \ne 0$, $n \in \mathbb N_0$, we obtain
$$
  	\int_0^{\pi} h_2(t) \cos \sqrt{\mu_n} t \, dt = 0, \quad \int_0^{\pi} h_1(t) \sin \sqrt{\nu_n} t \, dt = 0, \quad n \in \mathbb N_0.
$$
The systems $\{ \cos \sqrt{\mu_n} t \}_{n \in \mathbb N_0}$ and $\{ \sin \sqrt{\nu_n} t \}_{n \in \mathbb N_0}$ are complete in $L_2(0, \pi)$ (see \cite{FY01}), so
$h_1(t) = 0$ and $h_2(t) = 0$ for a.e. $x \in (0, \pi)$. Thus, the system $\{ v_{nk} \}$ is complete in $\mathcal H$. 
Hence it is a Riesz basis.

\end{proof}

In view of Lemmas~\ref{lem:asymptv} and \ref{lem:riesz}, the relation \eqref{scal} provides the $l_2$-sequence of the coordinates
of the unknown vector functon $f$ with respect to the Riesz basis. One can uniquely determine the function $f$ by these coordinates.
Thus, we obtain the following algorithm for the solution of the IP.

\medskip

{\bf Algorithm.} Given the potentials $q_j$, $j = \overline{2, m}$, and the eigenvalues $\{ \la_{nk} \}_{n \in \mathbb N, \, k = 1, 2}$ of the problem $L$.

\begin{enumerate}
\item Find $\om_j = \frac{1}{2} \int\limits_0^{\pi} q_j(t) \, dt$, $j = \overline{2, m}$, the value of $\hat \om$ from \eqref{asymptrho1} and 
$\om_1 = m \hat \om - \sum\limits_{j = 2}^m \om_j$.
\item Construct the Weyl functions $M_j(\la)$ of the boundary value problems $L_j$, $j = \overline{2, m}$.
\item Calculate $g_{nk}$, $n \in \mathbb N$, $k = 1, 2$, via \eqref{relM}, then $v_{nk}$ and $f_{nk}$, $n \in \mathbb N_0$, $k = 1, 2$, via \eqref{defv}, \eqref{deff1}, \eqref{deff2}, \eqref{defvf0}.
\item Construct the function $f(t)$, using its coefficients $\{ f_{nk} \}_{n \in \mathbb N_0, \, k = 1, 2}$ 
by the Riesz basis $\{ v_{nk} \}_{n \in \mathbb N_0, \, k = 1, 2}$, find $K(t)$ and $N(t)$.
\item Construct the Weyl function $M_1(\la) = -\dfrac{S_1'(\pi, \la)}{S_1(\pi, \la)}$, using \eqref{intS}.
\item Solve the classical inverse problem by the Weyl function and find $q_1(x)$, $x \in (0, \pi)$.
\end{enumerate}

\medskip

Alternatively to the steps 5 and 6, one can construct the potential $q_1$ by the Cauchy data $\mathcal K(\pi, t) = \int_0^t K(s) \, ds$,
$\frac{\partial}{\partial x}\mathcal K(x, t)_{|x = \pi} = N(t)$, using the methods from \cite{RS92}. 

\medskip

\begin{remark}
Note that in this section, we have used the assumption (iv) only for $j = 1$. It will be required for $j = \overline{2, m}$ in the next section
for the proof of local solvability and stability. In fact, we can absolutely omit the assumption (iv) for the algorithm.
Indeed, $g_{n2}$ stands only in the denominators in the formulas \eqref{defv}, \eqref{deff2}, so the estimate $\dfrac{1}{g_{n2}} = O(n^{-2})$
is sufficient for the purposes of this section. One can even take $g_{n2} = \iy$ (this corresponds to $S_1(\pi, \la_{n2}) = 0$)
at the steps 3--6 of the algorithm. The proof of Lemma~\ref{lem:riesz} can be slightly modified to take this case into account.
\end{remark}

\begin{remark}
Suppose that the assumption (i) is violated, i.e. the subsequence $\{ \la_{nk} \}_{n \in \mathbb N, \, k = 1, 2}$ contains 
multiple eigenvalues. 
There can be only a finite number of them because of the asymptotics \eqref{asymptrho1}, \eqref{asymptrho2}.
For instance, let $\la_0 \in \{ \la_{nk} \}_{n \in \mathbb N, \, k = 1, 2}$ be a double eigenvalue of $L$.
Then 
$$
  	\Delta(\la_0) = \frac{d}{d \la} \Delta(\la)_{|\la = \la_0} = 0.
$$
Using \eqref{defDelta}, we obtain
\begin{gather*}
  	\ga_1 S_1'(\pi, \la_0) + \ga_2 S_1(\pi, \la_0) = 0, \\
  	\ga_1 \frac{d}{d \la} S_1'(\pi, \la)_{|\la = \la_0} + \ga_2 \frac{d}{d \la} S_1(\pi, \la)_{\la = \la_0} + \ga_3 S_1'(\pi, \la_0) + \ga_4 S_1(\pi, \la_0) = 0,
\end{gather*}
where $\ga_i$, $i = \overline{1, 4}$, are the constants, which can be easily found from \eqref{defDelta}. Take the following couple 
of the vector functions, associated with the eigenvalue $\la_0 = \rho_0^2$:
$$
  	v_1 = \begin{bmatrix} \ga_1 \rho_0 \sin \rho_0 t \\ \ga_2 \cos \rho_0 t \end{bmatrix}, \quad
  	v_2 = \begin{bmatrix} \ga_1 \frac{d}{d \la} (\rho \sin \rho t)_{| \rho = \rho_0} + \ga_3 \rho_0 \sin \rho_0 t \\
  	                      \ga_2 \frac{d}{d \la} (\cos \rho t)_{|\rho = \rho_0} + \ga_4 \cos \rho_0 t \end{bmatrix}. 
$$
One can show that the function $H(\la)$ from the proof of Lemma~\ref{lem:riesz} has a double zero at $\la_0$.
Consequently, $v_1$ and $v_2$ together with the vector functions, constructed via \eqref{defv} for simple eigenvalues, form a Riesz basis.
So our method can be applied for multiple eigenvalues with minor modifications.
\end{remark}

\begin{remark}
In view of the previous remarks, the most crucial restriction among (i)-(v) is (iii), while the other assumptions are just technical.
Indeed, let (iii) is violated, and $S_j(\pi, \la_{nk}) = 0$ for some fixed indices $j = \overline{2, m}$, $n \in \mathbb N$, $k = 1, 2$.
The relations \eqref{defDelta} and $\Delta(\la_{nk}) = 0$ imply that $S_p(\pi, \la_{nk}) = 0$ also for some $p \ne k$. 
Then the eigenvalue $\la_{nk}$ does not give us any information about the potential $q_1$.
\end{remark}

\bigskip

{\large \bf 4. Local solvability and stability}

\bigskip

Suppose that the potentials $q_j$, $j = \overline{2, m}$, and the eigenvalues $\{ \la_{nk} \}_{n \in \mathbb N, \, k = 1, 2}$ of the problem $L$ are given,
and the assumptions (i)-(v) are satisfied.
Let $\{ \tilde \la_{nk} \}_{n \in \mathbb N, \, k = 1, 2}$ be arbitrary real numbers, such that
\begin{equation} \label{epsrho}
 	\left( \sum_{n = 1}^{\iy} \sum_{k = 1, 2} (n (\rho_{nk} - \tilde \rho_{nk}))^2 \right)^{1/2} < \eps, \quad \tilde \rho_{nk} := \sqrt {\tilde \la_{nk}}.
\end{equation}
In this section, we will show, that if $\eps > 0$ is sufficiently small, the numbers $\{ \tilde \la_{nk} \}$ belong to the spectrum
of some boundary value problem $\tilde L$ of the same form as $L$, but with a different potential $\tilde q_1 \in L_2(0, \pi)$
(the other potentials coincide: $q_j = \tilde q_j$, $j = \overline{2, m}$). Moreover, the potential $\tilde q_1$ is sufficiently ``close''
to $q_1$ in $L_2$-norm, so the solution of the IP is stable. 
We agree that if a certain symbol $\gamma$ denotes an object related to $L$, then the corresponding symbol
$\tilde \gamma$ with tilde denotes the analogous object related to $\tilde L$.

By virtue of \eqref{epsrho}, the asymptotic formulas hold
\begin{align*}
  	\tilde \rho_{n1} & = n - \frac{1}{2} + \frac{\tilde {\hat \om}}{\pi n} + \frac{\tilde \varkappa_{n1}}{n},  \\
    \tilde \rho_{n2} & = n + \frac{\tilde z_1}{\pi n} + \frac{\tilde \varkappa_{n2}}{n},
\end{align*}
where $n  \in \mathbb N$, $\{ \tilde \varkappa_{nk} \} \in l_2$, $k = 1, 2$, $\tilde {\hat \om} = \hat \om$, $\tilde z_1 = z_1$.
Moreover, for sufficiently small $\eps$ the values $\{ \tilde \la_{nk} \}_{n \in \mathbb N, \, k = 1, 2}$ are distinct and positive.
Put $\tilde \om := \om$ and 
\begin{equation} \label{defgt}
   \tilde g_{nk} := \sum_{j = 2}^m \frac{S_j'(\pi, \tilde \la_{nk})}{S_j(\pi, \tilde \la_{nk})}, \quad n \in \mathbb N, \: k = 1, 2.
\end{equation}
Note that the assumptions (iii) and (iv) hold for $\{ \tilde \la_{nk} \}$, if $j = \overline{2, m}$ and $\eps$ is sufficiently small. Consequently,
$g_{nk} \ne \iy$, and we can derive the asymptotic relations
\begin{align*}
   \tilde g_{n1} & = \sum_{j = 2}^m (\om_j - \hat \om + \eta_{nj1}) = \om - \hat \om + \eta_{n11}, \\
   \tilde g_{n2} & = \sum_{j = 2}^m \frac{n^2}{\om_j - z_1} (1 + \eta_{nj2}) = \frac{n^2}{z_1 - \om_1} (1 + \eta_{n12}),
\end{align*}
where $\{ \eta_{njk} \} \in l_2$. Here we have used the fact that $z_1$ is a root of the characteristic polynomial $P(z)$.

Construct $\tilde v_{nk}$ and $\tilde f_{nk}$ by formulas, similar to \eqref{defv}, \eqref{deff1}, \eqref{deff2}, \eqref{defvf0}.
Clearly, these sequences satisfy the assertion of Lemma~\ref{lem:asymptv}.
Using \eqref{epsrho}, we obtain the following result.

\begin{lem}
There exists $\eps_0 > 0$, such that for every $\eps \in (0, \eps_0]$ the following estimates hold
$$
    \left ( \sum_{n = 1}^{\iy} (g_{n1} - \tilde g_{n1})^2 \right)^{1/2} < C \eps, \quad 
    \left( \sum_{n = 1}^{\iy}  (n^{-2} (g_{n2} - \tilde g_{n2}))^2 \right)^{1/2} < C \eps, 
$$
$$
   \left( \sum_{n = 0}^{\iy} \sum_{k = 1, 2} (f_{nk} - \tilde f_{nk})^2 \right)^{1/2} < C \eps, \quad
   \left( \sum_{n = 0}^{\iy} \sum_{k = 1, 2} \| v_{nk} - \tilde v_{nk} \|_{\mathcal H}^2 \right)^{1/2} < C \eps.
$$
where $C$ is some constant depending only on $L$ and $\eps_0$.
\end{lem}

Hence for sufficiently small $\eps > 0$, the sequence $\{ \tilde v_{nk} \}_{n \in \mathbb N, \, k = 1, 2}$ is a Riesz basis in $\mathcal H$,
and the sequence $\{ \tilde f_{nk} \}_{n \in \mathbb N, \, k = 1, 2}$ belongs to $l_2$.
Further we need the following abstract result.

\begin{lem} \label{lem:abstract}
Let $\{ v_n \}_{n \in \mathbb N}$ and $\{ \tilde v_n \}_{n \in \mathbb N}$ be Riesz bases in a Hilbert space $H$,
$f \in H$, $f_n = (f, v_n)_H$, $n \in \mathbb N$, and  
\begin{equation} \label{condvf}
   	\left( \sum_{n = 1}^{\iy} \| v_n - \tilde v_n \|_H^2 \right)^{1/2} < \eps, \quad \left( \sum_{n = 1}^{\iy} |f_n - \tilde f_n|^2 \right)^{1/2} < \eps.
\end{equation}
Let $\tilde f$ be the element of $H$ with coordinates $\tilde f_n = (\tilde f, \tilde v_n)_H$. Then $\| f - \tilde f \|_{H} < C \eps$,
where the constant $C$ depends only on $f$, $\{ v_n \}$ and $\eps_0$, if $\eps \in (0, \eps_0]$.
\end{lem}

\begin{proof}
For the Riesz bases $\{ v_n \}$ and $\{ \tilde v_n \}$ there exist biorthonormal bases
$\{ \chi_n \}$, $\{ \tilde \chi_n \}$, bounded linear invertible operators $A$, $\tilde A$
and an orthonormal basis $\{ e_n \}$, such that
$$
  	A e_n = v_n, \quad (A^*)^{-1} e_n = \chi_n, \quad \tilde A e_n = \tilde v_n, \quad (\tilde A^*)^{-1} e_n = \tilde \chi_n, \quad
  	n \in \mathbb N.
$$
Consequently, $v_n - \tilde v_n = (A - \tilde A) e_n$.
By virtue of \eqref{condvf}, we have
$$
  	\left( \sum_{n = 1}^{\iy} \| (A - \tilde A) e_n \|_H^2 \right)^{1/2} < \eps.
$$
Hence $\| A - A^* \|_{H \to H} \le \eps$. One can also show that 
\begin{equation} \label{estA*}
\| (A^*)^{-1} - (\tilde A^*)^{-1} \|_{H \to H} < C \eps,
\end{equation}
where the constant $C$ depends only on $A$ and $\eps_0$, if $\eps \in (0, \eps_0]$. The elements $f$ and $\tilde f$ can be
represented in the following form
$$
   f = \sum_{n = 1}^{\iy} f_n \chi_n, \quad \tilde f = \sum_{n = 1}^{\iy} \tilde f_n \tilde \chi_n.
$$
Consequently,
$$
   f - \tilde f = (A^*)^{-1} \sum_{n = 1}^{\iy} f_n e_n - (\tilde A^*)^{-1} \sum_{n = 1}^{\iy} \tilde f_n e_n =
   (A^*)^{-1} \sum_{n = 1}^{\iy} (f_n - \tilde f_n) e_n + ((A^*)^{-1} - (\tilde A^*)^{-1}) \sum_{n = 1}^{\iy} \tilde f_n e_n.
$$
Using the estimates \eqref{condvf} and \eqref{estA*}, we arrive at the assertion of the lemma.
\end{proof}

Let $\tilde f(t) = \begin{bmatrix} \tilde N(t) \\ \tilde K(t) \end{bmatrix}$ be the vector function in $\mathcal H$, having
the coordinates $\{ \tilde f_{nk} \}_{n \in \mathbb N, \, k = 1, 2}$ with respect to the Riesz basis $\{ \tilde v_{nk} \}_{n \in \mathbb N, \, k = 1, 2}$.
(Strictly speaking, they are the coordinates with respect to the biorthogonal basis).
By virtue of Lemma~\ref{lem:abstract}, the estimates hold
\begin{equation} \label{estKN}
  	\| K(t) - \tilde K(t) \|_{L_2} < C \eps, \quad \| N(t) - \tilde N(t) \|_{L_2} < C \eps.
\end{equation}
Construct the functions
\begin{align} \label{intSt} 
   \tilde S_1(\pi, \la) = \frac{\sin \rho \pi}{\rho} - \frac{\om \cos \rho \pi}{\rho^2} + \frac{1}{\rho^2} \int_0^{\pi} \tilde K(t) \cos \rho t \, dt, \\ \label{intSpt}
   \tilde S_1'(\pi, \la) = \cos \rho \pi + \frac{\om \sin \rho \pi}{\rho} + \frac{1}{\rho} \int_0^{\pi} \tilde N(t) \sin \rho t \, dt.
\end{align}    
Since
$$
   \int_0^{\pi} \tilde K(t) \, dt = \om,
$$
the functions $\tilde S_1(\pi, \la)$ and $\tilde S_1'(\pi, \la)$ are entire in $\la$. Denote their zeros by $\{ \tilde \mu_n \}_{n \in \mathbb N}$ and
$\{ \tilde \nu_n \}_{n \in \mathbb N_0}$, respectively. The following lemma asserts, that the sequences
$\{ \tilde \mu_n \}_{n \in \mathbb N}$ and $\{ \tilde \nu_n \}_{n \in \mathbb N_0}$ are sufficiently ``close'' to the zeros
$\{ \mu_n \}_{n \in \mathbb N}$ and $\{ \nu_n \}_{n \in \mathbb N_0}$ of the functions $S_1(\pi, \la)$ and $S_1'(\pi, \la)$, respectively.

\begin{lem} \label{lem:munu}
The following estimates take place
\begin{equation} \label{estmu}
    \left( \sum_{n = 1}^{\iy} (\mu_n - \tilde \mu_n)^2 \right)^{1/2} < C \eps, \quad 
    \left( \sum_{n = 0}^{\iy} (\nu_n - \tilde \nu_n)^2 \right)^{1/2} < C \eps.
\end{equation}
\end{lem}

The proof of Lemma~\ref{lem:munu} is based on the estimates \eqref{estKN}. One can easily check, that 
for sufficiently small $\eps$, the numbers $\{ \tilde \mu_n \}_{n \in \mathbb N}$ and $\{ \tilde \nu_n \}_{n \in \mathbb N_0}$
are real. So we can apply the following result by G.~Borg (see \cite{Borg}, \cite[Section 1.8]{FY01}).

Consider the boundary value problem $L_1$, defined by \eqref{Lj}, and the problem $L_1^0$ for the same equation   
with the boundary conditions $y_1(0) = y_1'(\pi) = 0$. Obviously, the sequences $\{ \mu_n \}_{n \in \mathbb N}$ and $\{ \nu_n \}_{n \in \mathbb N_0}$
are the spectra of $L_1$ and $L_1^0$, respectively.

\begin{thm}[G. Borg]
For the boundary value problems $L_1$ and $L_1^0$, there exists $\eps_0 > 0$ (which depends on $L_1$, $L_1^0$), such that
if real numbers $\{ \tilde \mu_n\}_{n \in \mathbb N}$ and $\{ \tilde \nu_n \}_{n \in \mathbb N_0}$ satisfy the condition \eqref{estmu} for $C \eps \le \eps_0$,
then there exists a unique real function $\tilde q_1 \in L_2(0, \pi)$, for which the numbers $\{ \tilde \mu_n \}_{n \in \mathbb N}$ and
$\{ \tilde \nu_n \}_{n \in \mathbb N_0}$ are the eigenvalues of $\tilde L_1$ and $\tilde L_1^0$, respectively. Moreover, 
$$
  	\| q_1 - \tilde q_1 \|_{L_2} < C_1 \eps.	
$$
\end{thm}

Let $\tilde L$ be the boundary value problem of the form \eqref{eqv}, \eqref{BC}, \eqref{cont}, \eqref{kirch} with
the potentials $\tilde q_1$ and $\tilde q_j = q_j$, $j = \overline{1, m}$.
It remains to prove that $\{ \tilde \la_{nk} \}_{n \in \mathbb N, \, k = 1, 2}$ are eigenvalues of $\tilde L$.
Indeed, one can easily show that the assumptions (i)-(v) are valid for
the given values $\{ \tilde \la_{nk} \}_{n \in \mathbb N, \, k = 1, 2}$ and the problem $\tilde L$ for sufficiently small $\eps$.
The characteristic functions of the problems $\tilde L_1$ and $\tilde L_1^0$ coincide with $\tilde S_1(\pi, \la)$ and
$\tilde S_1'(\pi, \la)$, defined by \eqref{intSt}, \eqref{intSpt}. The functions $\tilde N$ and $\tilde K$ were constructed 
in such a way, that $-\dfrac{\tilde S_1'(\pi, \la_{nk})}{\tilde S_1(\pi, \la_{nk})} = \tilde g_{nk}$, $n \in \mathbb N$, $k = 1, 2$.
Together with \eqref{defgt} and the assumption (iii) this implies that $\{ \tilde \la_{nk} \}_{n \in \mathbb N, \, k = 1, 2}$ are zeros of the 
characteristic function 
\begin{equation*}
   \tilde \Delta(\la) = \tilde S_1'(\pi, \la) \prod_{j = 2}^m S_j(\pi, \la) + \tilde S_1(\pi, \la) \left( \sum_{j = 2}^m S_j'(\pi, \la) 
   \prod_{\substack{k = 2 \\ k \ne j}}^m S_k(\pi, \la) \right)
\end{equation*}
of the problem $\tilde L$. Thus, we have proved the following theorem.

\begin{thm} \label{thm:loc}
For every boundary value problem $L$, satisfying the assumptions (i)-(v), there exists $\eps_0 > 0$,
such that for arbitrary real numbers $\{ \tilde \la_{nk} \}_{n \in \mathbb N, \, k = 1, 2}$,
satisfying \eqref{epsrho} for $\eps \in (0, \eps_0]$, there exists a unique real function $q_1 \in L_2(0, \pi)$, being the solution of the IP
for $\{ \tilde \la_{nk} \}_{n \in \mathbb N, \, k = 1, 2}$ and $q_j$, $j = \overline{2, m}$.
The following estimate holds
$$
    \| q_1 - \tilde q_1 \|_{L_2} < C \eps,
$$
where the constant $C$ depends only on $L$ and $\eps_0$.
\end{thm}

Theorem~\ref{thm:loc} gives the local solvability and the stability for the solution of the IP.

\begin{remark}
One can obtain a similar result for the more general case, when not only the eigenvalues, but also the potentials $q_j$, $j = \overline{2, m}$,
are perturbed:
$$
  	\| q_j - \tilde q_j \|_{L_2} < \eps, \quad j = \overline{2, m}.
$$ 
\end{remark}

\medskip

{\bf Acknowledgments}. This work was supported by the President grant MK-686.2017.1 
and by Grants 15-01-04864 and 16-01-00015 of the Russian Foundation for Basic Research.

\medskip

\noindent Natalia Pavlovna Bondarenko \\
Department of Applied Mathematics, Samara National Research University, \\
34, Moskovskoye Shosse, Samara 443086, Russia, 
Department of Mechanics and Mathematics, Saratov State University, \\
Astrakhanskaya 83, Saratov 410012, Russia, \\
e-mail: {\it BondarenkoNP@info.sgu.ru}

\end{document}